\documentclass[12pt,a4paper]{amsart}
\usepackage{mathrsfs}
\usepackage{amssymb,amsmath,amsthm,color}
\usepackage{graphicx,mcite}
\usepackage{hyperref}
\usepackage{url}
\usepackage{setspace}
\newcommand{\loc}{\mbox{loc}}

\newtheorem{theorem}{Theorem}[section]

\newtheorem{proposition}[theorem]{Proposition}

\theoremstyle{definition}

\newtheorem{remark}[theorem]{Remark}

\numberwithin{equation}{section}

\begin{document}

\title[Injectivity of twisted spherical means]
{Coxeter system of planes are sets of injectivity for the twisted spherical means on $\mathbb C^n$}
%\runningtitle{Sets of Injectivity for Weighted Twisted Spherical Means}

\author{R. K. Srivastava}
 %Address of record for the research reported here
\address{School of Mathematics, Harish-Chandra Research Institute, Allahabad, India 211019.}
%\curraddr{Department of Mathematics and Statistics,
%Case Western Reserve University, Cleveland, Ohio 43403}
\email{rksri@hri.res.in}
%\thanks will become a 1st page footnote.

%    General info

\date{\today}

%\dedicatory{Dedicated to Prof. E. M. Stein on the occasion of his 80th birthday.}


\begin{abstract}
In this article, we prove that any pair of perpendicular planes is a set of injectivity
for the twisted spherical means (TSM) for $L^p(\mathbb C^n)~(n\geq2)$ with $1\leq p\leq2.$
Then, we imitate that any Coxeter system of even number of planes is a set of
injectivity for the TSM for $L^p(\mathbb C^n).$ We further observe that a set
$S_R^{2n-1}\times\mathbb C$ is a set of injectivity for the TSM for a
ceratin class of functions on $\mathbb C^{n+1}.$

\end{abstract}

\maketitle

\section{Introduction}\label{Asection1}

Let $\mu_r$ be the normalized surface measure on sphere $S_r(x).$ Let $\mathscr F
\subseteq L^1_{\loc}(\mathbb R^n).$ We say that $S\subseteq\mathbb R^n$ is a set
of injectivity for the spherical means for $\mathscr F$ if for $f\in\mathscr  F$
with $f\ast\mu_r(x)=0, \forall r>0$ and $\forall x\in S,$ implies $f=0.$

\smallskip

In the work by Agranovsky and Quinto \cite{AQ}, it has been shown that sets of
non-injectivity for the spherical means for $C_c(\mathbb R^n)~(n\geq2)$ is contained
in the zero set of a certain harmonic polynomial. For non zero function
$f\in C_c(\mathbb R^n),$ write $S(f)=\{x\in\mathbb R^n: f\ast\mu_r(x)=0,~\forall r>0\}.$
Then they have proved that $S(f)=\bigcap_{k=0}^\infty Q_k^{-1}(0),$
where \[Q_k(x)=\int_{\mathbb R^n}f(y)|x-y|^{2k}dy.\]
Since all of $Q_k$ can not be identically zero, it follows that there exists the least
positive integer $k_o$ such that $Q_{k_o}\not\equiv0.$ Hence $\Delta Q_{k_o}=2k_o(2k_o+n-1)Q_{k_o-1}=0.$
That is, $S(f)\subseteq Q_{k_{o}}^{-1}(0).$ Since $Q_{k_o}$ is harmonic and
a harmonic polynomial can vanish to a Coxeter system of hyperplane intersecting
along a line, it follows that for $n>2,$ any Coxeter system of hyperplanes
intersecting along a line may fail to be a set of injectivity for the spherical
means on $\mathbb R^n.$ In general, any real cone $K\subset\mathbb R^n~(n>2)$ is
a set of injectivity for the spherical means for $C(\mathbb R^n)$ if and only if $K$
is not contained in the zero set of any homogeneous harmonic polynomial, (see \cite{AVZ}).

\smallskip

However, these results do not continue to hold for injectivity of the
twisted spherical means on $\mathbb C^n~(n\geq2),$ because of non-commutative
nature of underlying geometry of the Heisenberg group, (see \cite{BP, CCG, CDPT, H}).
The question, any odd Coxeter system of hyperplanes can be a set of injectivity for
the twisted spherical mean for $L^p(\mathbb C^n),$ is still unanswered.

\smallskip

Using a recent result of the author (\cite{Sri2}, p.10, Theorem 3.6) that any pair of perpendicular
lines in $\mathbb C$ is a set of injectivity for the TSM, we prove that any pair
of perpendicular planes in $\mathbb C^n~(n\geq2)$ is a set of injectivity for
the TSM for function in $L^p(\mathbb C^n),~1\leq p\leq2.$
Suppose $f\times\mu_r(z)=0,~\forall r>0$ and
$\forall~z\in\mathbb C^{n-1}\times\mathbb R\cup\mathbb C^{n-1}\times i\mathbb R.$
Then $f=0.$  Using this, we conclude that any Coxeter system of even number of planes
intersecting along a line is a set of injectivity for the TSM for $L^q(\mathbb C^n).$

\smallskip

Let $z=(z', z_{n+1})\in\mathbb C^{n+1}$ and $S_R^{2n-1}=\{z\in\mathbb C^n: |z|=R\}.$
In a result of Narayanan and Thangavelu \cite{NT1}, it has been proved that the spheres
centered at the origin are set of injectivity for the TSM on $\mathbb C^n.$ The author
has generalized their result for certain weighted twisted spherical
means, (see \cite{Sri}). In general, the question of set of injectivity for the twisted
spherical means (TSM) with real analytic weight is still open.
In view of result in \cite{NT1} that $S_R^{2n-1}$ is a set of injectivity
for the TSM on $\mathbb C^n,$ we prove that the set $S_R^{2n-1}\times\mathbb C$
is a set of injectivity for the twisted spherical means for the functions satisfying
$e^{\frac{1}{4}|z'|^2}f(z)\in L^p(\mathbb C^{n+1})$ with $1\leq p\leq\infty.$

\smallskip
On account of these results, we observe an embedding property of the sets of injectivity for
the TSM in higher dimensions. In the Euclidean set up, the sets of injectivity for the
spherical means on the unit sphere $S^{n-1}$ can  be embedded into the sets of injectivity
for the spherical means on $\mathbb R^n,$ (see \cite{AVZ}).

\smallskip

Since Laguerre function $\varphi_k^{n-1}$ is an eigenfunction of the special
Hermite operator $A=-\Delta_z+\frac{1}{4}|z|^2,$ with eigenvalue $2k+n,$ the
projection $f\times\varphi_k^{n-1}$ is also an eigenfunction of $A$ with
eigenvalue $2k+n.$ As $A$ is an elliptic operator and eigenfunction of an
elliptic operator is real analytic \cite{N}, the projection $f\times\varphi_k^{n-1}$
must be a real analytic function on $\mathbb C^n.$ By polar decomposition,
the conditions $f\times\mu_r(z)=0, ~\forall r>0$ is
equivalent to $f\times\varphi_k^{n-1}(z)=0, ~\forall k\in\mathbb Z_+,$
set of non-negative integers.
Therefore, any determining set for the real analytic functions is a set
of injectivity for the TSM on $L^p(\mathbb C^n)$ with $1\leq p\leq\infty.$
For example, let $\gamma(t)=r(t)e^{it},$ where $r(t)$ be a non-periodic real
analytic function on $[0,\infty)$ with $\lim_{t\rightarrow\infty}r(t)=0.$ Then
$\mathbb C^{n-1}\times\{\gamma(t):~t\in[0,\infty)\}$ is a set of injectivity
for the TSM for $L^p(\mathbb C^n)$ with $1\leq p\leq\infty.$ For details on
determining sets for real analytic functions, see \cite{PS, RS}.

\bigskip
In an interesting result, Courant and Hilbert (\cite{CH}, p. 699) had proved that if
the circular averages of a function $f$ which is even with respect to a line $L,$ vanishes
over all circles centered at points of $L,$ then $f\equiv0.$
As a consequence of this result, the circular averages of a function $f$ vanish over all circles
centered at points of $L$ if and only if $f$ is odd with respect to $L,$ (see \cite{AQ}, Lemma 6.3).
Hence, any line $L$ in $\mathbb R^2$ is not a set of injectivity for the spherical means for
the odd functions about $L.$

In 1996, Agranovsky and Quinto have completely characterized
the sets of injectivity for the Euclidean spherical means for compactly supported
functions on $\mathbb R^2.$ Their result says that the exceptional set for the sets
of injectivity is a very thin set which consists of a Coxeter system of lines union
finitely many points. Following theorem is their main result.

\begin{theorem}\label{Ath7}\cite{AQ}
A set $S\subset\mathbb R^2$ is a set of injectivity for the spherical means for
$C_c(\mathbb R^2)$ if and only if $S\nsubseteq\omega(\Sigma_N)\cup F,$ where $\omega$
is a rigid motion of $\mathbb R^2, \Sigma_N=\cup_{l=0}^{N-1}\{te^{\frac{i\pi l}{N}}: t\in\mathbb R\}$
is a  Coxeter system of $N$ lines and $F$ is a finite set in $\mathbb R^2.$
\end{theorem}
In particular, any closed curve is a set of injectivity for $C_c(\mathbb R^2).$
In fact, Agranovsky et al. \cite{ABK} further prove that the boundary of any bounded domain
in $\mathbb R^n~(n\geq2)$ is set of injectivity for the spherical means on
$L^p(\mathbb R^n),$ with $1\leq p\leq\frac{2n}{n-1}.$ For $p>\frac{2n}{n-1},$
unit sphere $S^{n-1}$ is an example of non-injectivity set in $\mathbb R^n.$
This result has been generalized for certain weighted spherical means, (see \cite{NRR}).
In general, the question of set of injectivity for the spherical means with real analytic
weight is still open. In \cite{NRR}, it has been shown that $S^{n-1}$ is a set of
injectivity for the spherical means with real analytic weight for the class of radial
functions on $\mathbb R^n.$

\smallskip

An analogue of Theorem \ref{Ath7} in the higher dimensions is still open and appeared as a
conjecture in their work \cite{AQ}. It says that the sets of non-injectivity for the Euclidean
spherical means are contained in a certain algebraic variety. Following is their question.

\smallskip

\noindent{\bf Conjecture \cite{AQ}.}
A set $S\subset\mathbb R^n$ is a set of injectivity for the spherical means for $C_c(\mathbb R^2)$
if and only if $S\nsubseteq\omega(\Sigma)\cup F,$ where $\omega$ is a rigid motion of $\mathbb R^n,$
$\Sigma$ is the zero set of a homogeneous harmonic polynomial and $F$ is an algebraic variety in
$\mathbb R^n$ of co-dimension at most $2.$

\smallskip

This conjecture remains unsolved, however a partial result related to this conjecture has
been proved by Kuchment et al. \cite{AK}. They also present a survey on the recent developments
towards the above conjecture. However, in this article, we observe that this conjecture does
not continue to hold for the spherical means on the Heisenberg group $\mathbb H^n=\mathbb C^n\times\mathbb R.$
In fact result on $\mathbb H^1$ (see \cite{Sri2}) is an adverse to the Euclidean result,
Theorem \ref{Ath7} on $\mathbb R^2.$

\smallskip

In more general, let $f$ be a non-zero function in $L^2(\mathbb C^n)$ and write
$S(f)=\{z\in\mathbb C^n: f\times\mu_r(z)=0, \forall r>0\}.$ Our main problem is
to describe completely the geometrical structure of $S(f)$ that would characterize
which ``sets" are set of injectivity for the TSM. For a type function
$f(z)=\tilde{a}(|z|)P(z)\in L^2(\mathbb C^n)\cap C(\mathbb C^n),$ where $P\in H_{p,q}$
is the space of bigraded homogeneous harmonic polynomials on $\mathbb C^n.$ In the
article \cite{Sri2},it has been shown that $S(f)=P^{-1}(0)\cup F,$ where $F$ is the
union of finitely many spheres centered at the origin. Since $P$ is harmonic, by maximal
principle $P^{-1}(0)$ can not contain the boundary of any bounded domain in $\mathbb C^n.$
Hence the boundary of any bounded domain would be a possible candidate for set of
injectivity for the TSM, (see \cite{AR, NT1, ST}).

\section{Notation and Preliminaries}\label{Asection2}
We define the twisted spherical means which arise in the study of spherical means
on Heisenberg group. The group $\mathbb H^n,$ as a manifold, is
$\mathbb C^n \times\mathbb R$ with the group law
\[(z, t)(w, s)=(z+w,t+s+\frac{1}{2}\text{Im}(z.\bar{w})),~z,w\in\mathbb C^n\text{ and }t,s\in\mathbb R.\]
Let $\mu_s$ be the normalized surface measure on the sphere $ \{ (z,0):~|z|=s\} \subset \mathbb H^n.$
The spherical means of a function $f$ in $L^1(\mathbb H^n)$ are defined by
\begin{equation} \label{Aexp22}
f\ast\mu_s(z, t)=\int_{|w|=s}~f((z,t)(-w,0))~d\mu_s(w).
\end{equation}
Thus the spherical means can be thought of as convolution operators. An important
technique in many problem on $\mathbb H^n$ is to take partial Fourier transform
in the $t$-variable to reduce matters to $\mathbb C^n$. Let
\[f^\lambda(z)=\int_\mathbb R f(z,t)e^{i \lambda t} dt\]
be the inverse Fourier transform of $f$ in the $t$-variable. Then a simple calculation
shows that
\begin{eqnarray*}
(f \ast \mu_s)^\lambda(z)&=&\int_{-\infty}^{~\infty}~f \ast \mu_s(z,t)e^{i\lambda t} dt\\
&=&\int_{|w| = s}~f^\lambda (z-w)e^{\frac{i\lambda}{2} \text{Im}(z.\bar{w})}~d\mu_s(w)\\
&=&f^\lambda\times_\lambda\mu_s(z),
\end{eqnarray*}
where $\mu_s$ is now being thought of as normalized surface measure
on the sphere $S_s(o)=\{z\in\mathbb C^n: |z|=s\}$ in $\mathbb C^n.$
Thus the spherical mean $f\ast \mu_s$ on the Heisenberg group can be
studied using the $\lambda$-twisted spherical mean $f^\lambda
\times_\lambda\mu_s$ on $\mathbb C^n.$ For $\lambda \neq 0,$
a further scaling argument shows that it is enough to study these
means for the case of $\lambda=1.$

Let $\mathscr F\subseteq L^1_{\loc}(\mathbb C^n).$  We say $S\subseteq\mathbb C^n$
is a set of injectivity for twisted spherical means for $\mathscr F$ if for
$f\in\mathscr F$ with $f\times\mu_r(z)=0, \forall r>0$ and $\forall z\in S,$
implies $f=0$ a.e.  The results on set of injectivity differ in the choice
of sets and the class of functions considered. We would like to refer to
\cite{AR, NT1, Sri}, for some results on the sets of injectivity for the TSM.

\bigskip

We need the following basic facts from the theory of bigraded
spherical harmonics (see \cite{T}, p.62 for details). We shall use
the notation $K=U(n)$ and $M=U(n-1).$ Then, $S^{2n-1}\cong K/M$ under
the map $kM\rightarrow k.e_n,$ $k\in U(n)$ and $e_n=(0,\ldots
,1)\in \mathbb C^n.$ Let $\hat{K}_M$ denote the set of all
equivalence classes of irreducible unitary representations of $K$
which have a nonzero $M$-fixed vector. It is known that each
representation in $\hat{K}_M$ has a unique nonzero $M$-fixed vector,
up to a scalar multiple.

For a $\delta\in\hat{K}_M,$ which is realized on $V_{\delta},$ let
$\{e_1,\ldots, e_{d(\delta)}\}$ be an orthonormal basis of
$V_{\delta}$ with $e_1$ as the $M$-fixed vector. Let
$t_{ij}^{\delta}(k)=\langle e_i,\delta (k)e_j \rangle ,$ $k\in K$
and $\langle , \rangle$ stand for the innerproduct on $V_{\delta}.$
By Peter-Weyl theorem, it follows that $\{\sqrt{d(\delta
)}t_{j1}^{\delta}:1\leq j\leq d(\delta ),\delta\in\hat{K}_M\}$ is an
orthonormal basis of $L^2(K/M)$ (see \cite{T}, p.14 for details).
Define $Y_j^{\delta} (\omega )=\sqrt{d(\delta )}t_{j1}^{\delta}(k),$
where $\omega =k.e_n\in S^{2n-1},$ $k \in K.$ It then follows that
$\{Y_j^{\delta}:1\leq j\leq d(\delta ),\delta\in \hat{K}_M, \}$
forms an orthonormal basis for $L^2(S^{2n-1}).$

For our purpose, we need a concrete realization of the representations in
$\hat{K}_M,$ which can be done in the following way. See \cite{Ru}, p.253,
for details. For $p,q\in\mathbb Z_+$, let $P_{p,q}$ denote the space of all
bigraded polynomials $P$ in $z$ and $\bar{z}$ of the form
\[P(z)=\sum_{|\alpha|=p}\sum_{|\beta|=q}c_{\alpha\beta} z^\alpha\bar{z}^\beta.\]
Let $H_{p,q}=\{P\in P_{p,q}:\Delta P=0\}.$ The elements of $H_{p,q}$ are
called the bigraded solid harmonics on $\mathbb C^n.$ The group  $K$
acts on $H_{p,q}$ in a natural way. It is easy to see that the space
$H_{p,q}$ is $K$-invariant. Let $\pi_{p,q}$ denote the corresponding
representation of $K$ on $H_{p,q}.$ Then representations in
$\hat{K}_M$ can be identified, up to unitary equivalence, with the
collection $\{\pi_{p,q}: p,q \in \mathbb Z_+\}.$

Define the bigraded spherical harmonic by $Y_j^{p,q}(\omega)=\sqrt{d(p,q )}t_{j1}^{p,q}(k).$
Then $\{Y_j^{p,q}:1\leq j\leq d(p,q),p,q \in \mathbb Z_+ \}$ forms an orthonormal basis for
$L^2(S^{2n-1}).$ Therefore, for a continuous function $f$ on $\mathbb C^n,$ writing
$z=\rho \,\omega,$ where $\rho>0$ and $\omega \in S^{2n-1},$ we can expand the function $f$
in terms of spherical harmonics as
\begin{equation}\label{ABexp4}
f(\rho\omega)=\sum_{p,q\geq0}\sum_{j=1}^{d(p,q)}a_j^{p,q}(\rho)Y_j^{p,q}(\omega),
\end{equation}
where the series on the right-hand side converges uniformly on every compact set
$K\subseteq\mathbb C^n.$ The functions $a_j^{p,q}$ are called the spherical
harmonic coefficients of $f$ and function $a^{p,q}(\rho)Y^{p,q}(\omega)$ is
known as the type function.

\smallskip
We also need an expansion of functions on $\mathbb C^n$ in terms of Laguerre
functions $\varphi_k^{n-1}$'s, which is know as special Hermite expansion.
The special Hermite expansion is a useful tool in the study of convolution
operators and is related to the spectral theory of sub-Laplacian on the
Heisenberg group $H^n$. However, more details can be found in \cite{T}.

\smallskip

For $\lambda\in\mathbb R^*=\mathbb R\setminus\{0\}$, let
$\pi_\lambda$ be the unitary representation of $H^n$ on $L^2(\mathbb
R^n)$ given by
$$\pi_\lambda(z,t)\varphi(\xi)=e^{i\lambda t}e^{i\lambda(x.\xi+\frac{1}{2}x.y)}\varphi(\xi+y),
\varphi\in L^2(\mathbb R^n).$$ A celebrated theorem of Stone and von Neumann
says that $\pi_\lambda$ is  irreducible and up to unitary equivalence
$\{\pi_\lambda:~\lambda\in\mathbb R\}$ are all the infinite dimensional unitary
irreducible representations of $H^n.$
Let \[T=\dfrac{\partial}{\partial t},X_j=\dfrac{\partial}{\partial
x_j}+\dfrac{1}{2}y_j\dfrac{\partial}{\partial t}~,
Y_j=\dfrac{\partial}{\partial
y_j}-\dfrac{1}{2}x_j\dfrac{\partial}{\partial t},~j=1,2,\ldots,n.\]
Then  $\{T,X_j,Y_j:~j=1,\ldots,n\}$ is a  basis for the  Lie Algebra
$\mathfrak h^n$ of all left invariant vector fields on $H^n.$
Define $\mathcal L=-\sum_{j=1}^n(X_j^2+Y_j^2),$ the second order
differential operator which is known as the sub-Laplacian of $H^n.$
The representation $\pi_\lambda$ induces a representation
$\pi_\lambda^*$ of $\mathfrak h^n,$ on the space of $C^\infty$
vectors in $L^2(\mathbb R^n)$ is defined by
\[\pi_\lambda^*(X)f=\left.\frac{d}{dt}\right\vert_{t=0}\pi_\lambda(\exp tX)f.\]
An easy calculation shows that
$\pi^*(X_j)=i\lambda x_j,~\pi^*(Y_j)=\dfrac{\partial}{\partial x_j},~j=1,2,\ldots,n$.
Therefore, $\pi_\lambda^*(\mathcal
L)=-\Delta_x+\lambda^2|x|^2=:H(\lambda),$ the scaled Hermite
operator. The eigenfunction of $H(\lambda)$ are given by
$\phi_\alpha^\lambda(x)=|\lambda|^{\frac{n}{4}}\phi_\alpha(\sqrt{|\lambda|} x), ~\alpha\in\mathbb Z_+^n,$
where $\phi_\alpha$ are the Hermite functions on $\mathbb R^n.$
Since
$H(\lambda)\phi_\alpha^\lambda=(2|\lambda|+n)|\lambda|\phi_\alpha^\lambda.$
Therefore,
\[\mathcal L\left(\pi_\lambda(z,t)\phi_\alpha^\lambda,\phi_\beta^\lambda\right)
=(2|\lambda|+n)|\lambda|\left(\pi_\lambda(z,t)\phi_\alpha^\lambda,\phi_\beta^\lambda\right).\]
Thus the entry functions
$\left(\pi_\lambda(z,t)\phi_\alpha^\lambda,\phi_\beta^\lambda\right),~
 \alpha,\beta\in\mathbb Z_+^n$ are eigenfunctions for $\mathcal L$.
As
$\left(\pi_\lambda(z,t)\phi_\alpha^\lambda,\phi_\beta^\lambda\right)
 =e^{i\lambda t}\left(\pi_\lambda(z)\phi_\alpha^\lambda,\phi_\beta^\lambda\right),$
these eigenfunctions are not in $L^2(H^n).$ However for a fix $t,$
they are in $L^2(\mathbb C^n).$ Define  $\mathcal L_\lambda$ by
$\mathcal L\left(e^{i\lambda t}f(z)\right)=e^{i\lambda t}L_\lambda
f(z).$ Then the functions
\[\phi_{\alpha\beta}^\lambda(z)=
(2\pi)^{-\frac{n}{2}}\left(\pi_\lambda(z)\phi_\alpha^\lambda,\phi_\beta^\lambda\right),\]
are eigenfunction of the operator $L_\lambda$ with eigenvalue $2|\lambda|+n.$
The functions $\phi^\lambda_{\alpha\beta}$'s are called the special Hermite
functions and they form an orthonormal basis
for $L^2(\mathbb C^n)$ (see \cite{T}, Theorem 2.3.1, p.54). Thus,
for $g\in L^2(\mathbb C^n)$, we have the expansion
\[g=\sum_{\alpha,\beta}\left\langle g,\phi^\lambda_{\alpha\beta}\right\rangle\phi^\lambda_{\alpha\beta}.\]
To further simplify this expansion, let
$\varphi^{n-1}_{k,\lambda}(z)=\varphi^{n-1}_k(\sqrt{|{\lambda|}}z),$
the Leguerre function of degree $k$ and order $n-1$. The special
Hermite functions $\phi^\lambda_{\alpha\alpha}$ satisfy the relation
\begin{equation}\label{ACexp4}
\sum_{|\alpha|=k}\phi^\lambda_{\alpha,\alpha}(z)=(2\pi)^{-\frac{n}{2}}
|\lambda|^{\frac{n}{2}}\varphi^{n-1}_{k,\lambda}(z).
\end{equation}
Let $g$ be a function in $L^2(\mathbb C^n)$. Then $g$ can be expressed as
\[g(z)=(2\pi)^{-n}|\lambda|^n\sum_{k=0}^\infty g\times_\lambda\varphi_{k,\lambda}^{n-1}(z),\]
whenever $\lambda\in\mathbb R^*,$ (see \cite{T}, p.58). In particular, for $\lambda=1$, we have
\begin{equation}\label{ACexp16}
g(z)=(2\pi)^{-n}\sum_{k=0}^\infty g\times\varphi_k^{n-1}(z),
\end{equation}
which is called the special Hermite expansion for $g$.

\section{Sets of injectivity for the twisted spherical means}\label{Asection3}
In this section, we prove that any Coxeter system of even number of hyperplanes
intersecting along a line is a set of injectivity for the TSM for
$L^p(\mathbb C^n)~(n\geq2),$ for $1\leq p\leq2.$ Then, we prove that the set
$S_R^{2n-1}\times\mathbb C$ is a set of injectivity for the TSM for a ceratin
class of functions on $\mathbb C^{n+1}.$ In the first case, we deduce a density
result for $L^p(\mathbb C^n)$ with $2\leq p<\infty.$

\smallskip

In order to prove our main results, we need the following results from \cite{NT1, Sri2}.
\begin{theorem} \label{Ath3}\cite{NT1}
Let $f$ be a function on $\mathbb C^n$ such that $e^{\frac{1}{4}|z|^2}f(z)\in
L^p(\mathbb C^n),$ for $1\leq p\leq \infty.$ Suppose $f\times\mu_r(z)=0, \forall~r>0$
and $\forall~z\in\ S_R^{2n-1}.$ Then $f=0$ a.e. on $\mathbb C^n.$
\end{theorem}

\begin{remark}\label{Ark3}
For $\eta\in\mathbb C^n,$ define the left twisted translate by
\[\tau_\eta f(\xi)=f(\xi-\eta)e^{\frac{i}{2}\text{Im}(\eta.\bar\xi)}.\]
Then $\tau_{\eta}(f\times\mu_r)=\tau_{\eta}f\times\mu_r.$ Since the
function space considered as in the above Theorem \ref{Ath3} is not twisted
translation invariant, it follows that a sphere centered off the origin is not
set of injectivity for the TSM on $\mathbb C^n.$ The author has generalized
Theoren \ref{Ath3} for certain weighted twisted spherical means, (see \cite{Sri}).
\end{remark}

As the space $L^p(\mathbb C^n)$ is twisted translations invariant,
to prove any Coxeter system of even number of hyperplanes intersecting along a
line is a set of injectivity for the TSM for $L^p(\mathbb C^n),$ it is enough
to prove that the set $\mathbb C^{n-1}\times\Sigma_{2N},$ where
$\cup_{l=0}^{2N-1}\{te^{\frac{i\pi l}{2N}}: t\in\mathbb R\},$
is a set of injectivity for the TSM for $L^p(\mathbb C^n).$

 \begin{theorem}\label{Ath1}\cite{Sri2}
Let $f\in L^p(\mathbb C),$ for $1\leq p\leq2.$ Suppose $f\times\mu_r(z)=0, \forall~r>0$
and $\forall~z\in\Sigma_{2N}.$  Then $f=0$ a.e. on $\mathbb C.$
\end{theorem}
Using Theorem \ref{Ath1}, we prove the following result. Let $S_n=\mathbb C^{n-1}\times\Sigma_{2N}.$
\begin{theorem}\label{Ath2}
Let $f\in L^p(\mathbb C^n),$ for $1\leq p\leq2.$ Suppose $f\times\mu_r(z)=0,~\forall r>0$
and $\forall~z\in S_n.$ Then $f=0$ a.e. on $\mathbb C^n.$
\end{theorem}
\begin{proof}
Since $f\times\mu_r(z)=0, ~\forall r>0,$  by polar decomposition, it follows that
$f\times\varphi_k^{n-1}(z)=0, ~\forall k\in\mathbb Z_+.$ Given that $f\in L^p(\mathbb C^n).$
By convolving $f$ with a right and radial compactly supported smooth approximate identity,
we can assume $f\in L^2(\mathbb C^n).$ In order to prove the result on $\mathbb C^n,$ We first
prove the result on $\mathbb C^2$ and then by induction hypothesis on $n,$ we deduce it for
$\mathbb C^n.$ Since
\[\varphi_k^1(z_1,z_2)=\sum_{\beta_1+\beta_2=k}\varphi_{\beta_1}^0(z_1)\varphi_{\beta_2}^0(z_2).\]
Therefore, we can write
\begin{eqnarray}\label{Aexp31}
f\times\varphi_k^{1}(z_1,z_2)&=&\sum_{\beta_1+\beta_2=k}\int_{\mathbb C^2}f(z_1-w_1,z_2-w_2)
\varphi_{\beta_1}^0(w_1)\varphi_{\beta_2}^0(w_2)\nonumber\\
&&\times e^{\frac{i}{2}\text{Im}(z_1.\bar{w_1}+z_2.\bar{w_2})}dw_1dw_2,\nonumber\\
&=&\sum_{\beta_1+\beta_2=k}\int_{\mathbb C}f\times_{2}\varphi_{\beta_2}(z_1-w_1,z_2)
\varphi_{\beta_1}^0(w_1)e^{\frac{i}{2}\text{Im}(z_1.\bar{w_1})}dw_1\nonumber \\
&=&\sum_{\beta_1+\beta_2=k}\int_{\mathbb C}F_{z_2,\beta_2}(z_1-w_1)
\varphi_{\beta_1}^0(w_1)e^{\frac{i}{2}\text{Im}(z_1.\bar{w_1})}dw_1 \nonumber\\
&=&\sum_{\beta_1+\beta_2=k}F_{z_2,\beta_2}\times_1\varphi_{\beta_1}^0(z_1),\nonumber\\
\end{eqnarray}
where the function $F_{z_2,\beta_2}$ is defined by
\[F_{z_2,\beta_2}(z_1)=\int_{\mathbb C}f(z_1,z_2-w_2)
\varphi_{\beta_2}^0(w_2) e^{\frac{i}{2}\text{Im}(z_2.\bar{w_2})}dw_2.\]
For fixed $z_2,$ using the Minkowski integral inequality, it can show that
the function $F_{z_2,\beta_2}\in L^2(\mathbb C).$ By the hypothesis,
$f\times\varphi_k^{1}(z_1,z_2)=0,~\forall (z_1,z_2)\in S_2$ and
$\forall k\in\mathbb Z_+.$ Therefore, from equation (\ref{Aexp31}), we can write
\[\sum_{\beta_1+\beta_2=k}F_{z_2,\beta_2}\times_1\varphi_{\beta_1}^0(z_1)=0,~
\forall (z_1,z_2)\in S_2 \text{ and }\forall k\in\mathbb Z_+.\]
As the above equation is valid for each $k\in\mathbb Z_+,$ it follows that the sum
over each of the diagonal $\beta_1+\beta_2=k$ is zero. Using the facts that set
$\{\varphi_{\beta_1}^0:~\beta_1\in\mathbb Z_+\}$  form an orthogonal basis for
$L^2(\mathbb C)$ and $S_2=\mathbb C\times\Sigma_{2N},$ it follows that
\[F_{z_2,\beta_2}\times_1\varphi_{\beta_1}^0(z_1)=0,~\forall \beta_1, \beta_2\in\mathbb Z_+.\]
By equations (\ref{Aexp31}), we have
\[f\times\left(\varphi_{\beta_1}^0\varphi_{\beta_2}^0\right)(z_1,z_2)=0, \forall (z_1,z_2)\in S \text{ and }
\forall \beta_1, \beta_2\in\mathbb Z_+.\]
Now, we can write,
\begin{eqnarray*}
f\times\left(\varphi_{\beta_1}^0\varphi_{\beta_2}^0\right)(z_1,z_2)
&=& \int_{\mathbb C}G_{z_1,\beta_1}(z_2-w_2)
\varphi_{\beta_1}^0(w_2)e^{\frac{i}{2}\text{Im}(z_2.\bar{w_2})}dw_1 \nonumber\\
&=&G_{z_1,\beta_1}\times_{2}\varphi_{\beta_2}^0(z_2). \nonumber\\
\end{eqnarray*}
By the given condition, we have
$G_{z_1,\beta_1}\times_{2}\varphi_{\beta_2}^0(z_2)=0,\forall~\beta_1, \beta_2\in\mathbb Z_+.$
For each fixed $\beta_1,$ in view of Theorem \ref{Ath1}, we can conclude that
$G_{z_1,\beta_1}(z_2)=0,~\forall (z_1, z_2)\in\mathbb C^2.$ That is,
$f\times_1\varphi_{\beta_1}(z_1,z_2)=0,~\forall \beta_1\in\mathbb Z_+.$
Therefore,
\[f\times\varphi_k^{1}(z_1,z_2)=\sum_{\beta_1+\beta_2=k}\int_{\mathbb C}f\times_{1}\varphi_{\beta_1}(z_1,z_2-w_2)
\varphi_{\beta_1}^0(w_2)e^{\frac{i}{2}\text{Im}(z_2.\bar{w_2})}dw=0,\]
for all $k\in\mathbb Z_+.$ Hence, we conclude that $f=0,$ a.e. $\mathbb C^2.$
In order to prove the result for $n>2,$ we use the induction hypothesis on $n.$
Suppose the result is true for $n-1$ with $n>2.$ That is, the set
$S_{n-1}=\mathbb C^{n-2}\times\Sigma_{2N}$ is a set of
injectvity for the TSM for $L^2(\mathbb C^{n-1}).$ Let $k=\beta_1+\beta_2+\cdots+\beta_n
=\beta_1+|\gamma|$ and $z=(z_1,z_2,\ldots,z_n)=(z_1, z').$ Then, as similar to $\mathbb C^2$ case,
we can write
\[f\times\varphi_k^{n-1}(z_1,z')=\sum_{\beta_1+|\gamma|=k}F_{z',\gamma}\times_1\varphi_{\beta_1}^0(z_1).\]
Given that
\[f\times\varphi_k^{n-1}(z_1,z')=\sum_{\beta_1+|\gamma|=k}F_{z',\gamma}\times_1\varphi_{\beta_1}^0(z_1)=0
,~\forall (z_1,z')\in S_n \text{ and }\forall k\in\mathbb Z_+.\]
Since, the set $\{\varphi_{\beta_1}^0:~\beta_1\in\mathbb Z_+\}$ form an orthogonal basis
for $L^2(\mathbb C)$ and $S_n=\mathbb C^{n-1}\times\Sigma_{2N},$
it follows that $F_{z',\gamma}\times_1\varphi_{\beta_1}^0(z_1)=0, \forall \beta_1\in\mathbb Z_+$ and
$\forall ~|\gamma|\in\mathbb Z_+.$ This in turn implies that for each fixed $\gamma\in\mathbb Z_+^{n-1},$
we get $F_{z',\gamma}\times_1\varphi_{\beta_1}^0(z_1)=0, \forall \beta_1\in\mathbb Z_+$ and
$\forall (z_1,z')\in\mathbb C\times S_{n-1}.$ Once again using the orthogonality of
the set $\{\varphi_{\beta_1}^0:~\beta_1\in\mathbb Z_+\}$ in $L^2(\mathbb C),$
we conclude that
$F_{z',\gamma}(z_1)=0,\forall (z_1,z')\in\mathbb C\times S_{n-1}$ and $\forall\gamma\in\mathbb Z_+^{n-1}.$
Therefore, we have
\[\sum_{|\gamma|=k}F_{z',\gamma}(z_1)=\sum_{|\gamma|=k}\int_{\mathbb C^{n-1}}f(z_1,z'-w')
\Pi_{j=2}^n\varphi_{\beta_j}^0(w_j) e^{\frac{i}{2}\text{Im}(z'.\bar{w'})}dw'=0,\]
$\forall k\in\mathbb Z_+$ and $\forall z'\in S_{n-1}.$ By the assumption that $S_{n-1}$
is a set of injectivity for the TSM for $L^2(\mathbb C^{n-1}),$ we infer that
$f=0$ a.e. on $\mathbb C^n.$
\end{proof}

\begin{remark}\label{Ark1}
$(a).$ By the proof of Theorem \ref{Ath2}, it reveals that if $S$ is a set of injectivity
for the TSM for $L^p(\mathbb C),$ then the set $\mathbb C^{n-1}\times S$ will be a set of
injectivity for the TSM for $L^p(\mathbb C^n).$ That is, th sets of injectivity for the
TSM on $\mathbb C^n$ can be embedded into the sets injectivity for the TSM on $\mathbb C^{n+1}.$

\bigskip

$(b).$ We would like to mention that the question of Coxeter system of odd number
of hyperplanes intersecting along a line is a set of injectivity for the TSM for
$L^p(\mathbb C^n)~(n\geq2)$ can be answered, once we know that any Coxeter system
of odd number of lines can be a set of injectivity for the TSM for $L^p(\mathbb C).$
However, the later question on the complex place $\mathbb C$ is itself an open problem,
see \cite{Sri2}.
\end{remark}

Next, we prove that the set $S_{n+1}=S_R^{2n-1}\times\mathbb C$ is a set of injectivity
for the TSM for a ceratin class of functions on $\mathbb C^{n+1}.$ Let $z=(z',z_{n+1})\in
\mathbb C^{n+1}.$
\begin{theorem} \label{Ath4}
Let $f$ be a function on $\mathbb C^{n+1}$ such that $e^{\frac{1}{4}|z'|^2}f(z)\in L^p(\mathbb C^{n+1})$
for $1\leq p\leq\infty.$ Suppose $f\times\mu_r(z)=0,~\forall r>0$ and $\forall z\in S_{n+1}.$
Then $f=0$ a.e. on $\mathbb C^{n+1}.$
\end{theorem}

\begin{proof}
Let $k=\beta_1+\beta_2+\cdots+\beta_{n+1}=|\gamma|+ \beta_{n+1}.$
Then, we can write
\[f\times\varphi_k^{n}(z', z_{n+1})=\sum_{|\gamma|+ \beta_{n+1}=k}
F_{z',\gamma}\times_{(n+1)}\varphi_{\beta_{n+1}}^0(z_{n+1}).\]
By the given conditions, we have
\[f\times\varphi_k^{n}(z', z_{n+1})=\sum_{|\gamma|+ \beta_{n+1}=k}
F_{z',\gamma}\times_{(n+1)}\varphi_{\beta_{n+1}}^0(z_{n+1})=0,\]
for all $(z', z_{n+1})\in S_{n+1}\text{ and }\forall k\in\mathbb Z_+.$
Since, the set $\{\varphi_{\beta_1}^0: \beta_{n+1}\in\mathbb Z_+\}$ form an
orthogonal basis for $L^2(\mathbb C),$ it follows that
$F_{z',\gamma}\times_{(n+1)}\varphi_{\beta_1}^0(z_{n+1})=0, \forall \beta_{n+1}\in\mathbb Z_+$ and
$\forall ~|\gamma|\in\mathbb Z_+.$ This in turn implies that for each fixed $\gamma\in\mathbb Z_+^{n},$
$F_{z',\gamma}\times_{(n+1)}\varphi_{\beta_1}^0(z_{n+1})=0, \forall \beta_{n+1}\in\mathbb Z_+$ and
$\forall (z',z_{n+1})\in S_{n+1}.$ Once again using the orthogonality of
the set $\{\varphi_{\beta_{n+1}}^0:~\beta_{n+1}\in\mathbb Z_+\}$ in $L^2(\mathbb C),$
we conclude that
$F_{z',\gamma}(z_{n+1})=0,\forall \in S_{n+1}$ and $\forall\gamma\in\mathbb Z_+^{n}.$
Hence, we can write
\[\sum_{|\gamma|=k}F_{z',\gamma}(z_{n+1})=\sum_{|\gamma|=k}\int_{\mathbb C^{n}}f(z'-w',z_{n+1})
\Pi_{j=1}^n\varphi_{\beta_j}^0(w_j) e^{\frac{i}{2}\text{Im}(z'.\bar{w'})}dw'=0,\]
$\forall k\in\mathbb Z_+$ and $\forall z'\in S_R^{2n-1}.$ Therefore, in view of Theorem \ref{Ath3},
we infer that $f=0$ a.e. $\mathbb C^n.$
\end{proof}

Since the set $S_n=\mathbb C^{n-1}\times\Sigma_{2N}$ is a set of injectivity
for the TSM for $L^p(\mathbb C^n),$ with $1\leq p\leq2.$ As a dual problem,
it is natural to ask that $S_n$ is a set of density for $L^q(\mathbb C^n),$
for $2\leq q<\infty.$ Let $C_c^\sharp(\mathbb C)$ denote the space of radial
compactly supported continuous functions on $\mathbb C.$ Let
$\tau_zf(w)=f(z-w)e^{\frac{i}{2}\text{Im}(z.\bar w)}.$

\begin{proposition}\label{Aprop3}
The subspace $\mathscr F({S_n})
={\emph{Span}}\left\{\tau_zf: z\in S_n, f\in C_c^\sharp(\mathbb C^n) \right\}$
is dense in $L^q(\mathbb C),$ for $2\leq q<\infty.$
\end{proposition}

\begin{proof}
Let $\frac{1}{p}+\frac{1}{q}=1.$ Then $1\leq p\leq2.$ By Hahn-Banach theorem,
it is enough to show that $\mathscr F(S_n)^\bot=\{0\}.$ Let $g\in L^p(\mathbb C^n)$
be such that
\[\int_{\mathbb C}\tau_zf(w)g(w)dw=0, z\in\Sigma_{2N}, ~\forall f\in C_c^\sharp(\mathbb C^n).\]
That is, \[\overline{\bar g\times\bar f}(z)= f\times g(z)=0.\]
Let the support of $f$ be contained in $[0,t].$ Then by passing to the polar
decomposition, we get
\[\int_{r=0}^t\bar g\times\mu_r(z)\bar f(r)r^{2n-1}dr=0.\]
By differentiating the above equation, it follows that $\bar g\times\mu_t(z)=0, \forall t>0$
and $\forall z\in S_n.$ Thus by Theorem \ref{Ath2}, we conclude that $g=0$ a.e. on
$\mathbb C^n.$
\end{proof}

\begin{remark}\label{Ark2}
$(a).$
In the article by Agranovsky et al. \cite{AR}, it has been shown that boundary
of any bounded domain in $\mathbb C^n$ is a set of injectivity for the TSM
for a class of functions $f$ satisfying
$f(z)e^{(\frac{1}{4}+\epsilon)|z|^2}\in L^p(\mathbb C^n),$ for some $\epsilon>0$
and $1\leq p\leq\infty$. However to prove this result for $\epsilon=0$
is an open problem. The sphere $S_R^{2n-1}$ is an example with $\epsilon=0,$
as mentioned in Theorem \ref{Ath3}.

\smallskip

We are thinking to do away with exponential condition. Though it is not possible
for sphere, because of the relations
$\varphi^{n-1}_k\times\mu_r(z)=B(n,k)\varphi^{n-1}_k(r)\varphi^{n-1}_k(|z|).$
We are working for the real analytic curves $\gamma$ having non-constant
curvature can be the sets of injectivity for the TSM for $L^q(\mathbb C) $ with $1\leq q\leq2.$
We know that the spectral projections $Q_k=f\times\varphi^{n-1}_k$ is a real analytic
function on $\mathbb C^n. $ Using the Hecke-Bochner identity for the spectral
projections, we have derived a real analytic expansion for $Q_k$ in the article
\cite{Sri2} as
\[
Q_k(z)=\sum_{p=0}^kC_{k-p}^{p0}z^p\varphi_{k-p}^{p}(z)+\sum_{q=0}^\infty C_{k}^{0q}\bar z^q\varphi_{k}^{q}(z).
\]
Suppose $Q_k(z)=0,~\forall k\in\mathbb Z_+,$ and $\forall~z\in\gamma.$ For a curve $\gamma(t)=r(t)e^{it}$
of non-constant curvature, radius $r(t)$ will vary in an interval. The fact
that $Q_k$ is a real analytic function, $Q_k(\gamma(t))$ must be a real analytic
function, which vanishes over an interval. It is interesting to know that can $Q_k(\gamma(t))$ will
vanish for all $t\in\mathbb R.$ In particular, it can be possible when $r(t)$ is a non-periodic function.
For example, spiral $\gamma(t)=\left\{\left(e^{t}\cos t, e^{t}\sin t\right): t\in (-\infty,0]\right\}.$
Hence spiral is a set of injectivity for the TSM for $L^q(\mathbb C)$
with $1\leq q\leq\infty.$ For a more general example, let $\gamma(t)=r(t)e^{it},$
where $r(t)$ be a non-periodic real analytic function on $[0,\infty)$ with
$\lim_{t\rightarrow\infty}r(t)=0.$ Then $\gamma(t)$ is a set of injectivity for the
TSM for $L^q(\mathbb C)$ with $1\leq q\leq\infty.$ Moreover, $\gamma(t)$ is a
determining curve for any real analytic function on $\mathbb C.$

\smallskip

$(b).$ We are working for the set $S=\mathbb R\times\mathbb R\cup\mathbb R\times i\mathbb R$
can be set of injectivity for the TSM for $L^q(\mathbb C^2), ~1\leq q\leq2.$
We start with a small class of functions $f(z_1, z_2)=z_1^p\varphi^0_{\alpha_1}(z_1)h(z_2).$
Suppose $f\times\mu_r(z_1, z_2)=0,~\forall r>0$ and
$\forall~(z_1, z_2)\in\mathbb R\times\mathbb R\cup\mathbb R\times i\mathbb R.$
Then $f=0$ a.e. on $\mathbb C^2.$ If it goes to happen that $S$ is a set of
injectivity for the TSM for $L^q(\mathbb C^2)$ then it would be a surprise
contrast to the sets of injectivity for the Euclidean spherical means on
$\mathbb R^4,$ where the minimal dimension of a set of injectivity is three.
\end{remark}

\noindent{\bf Acknowledgements:}
The author wishes to thank E. K. Narayanan and S. Thangavelu for several
fruitful discussions, during my visit to IISc, Bangalore. The author would
also like to gratefully acknowledge the support provided by the University
Grant Commission and the Department of Atomic Energy, Government of India.


\bigskip
\begin{thebibliography}{11}

\bibitem{ABK} M. L. Agranovsky, C. Berenstein and P. Kuchment, {\em Approximation by spherical waves in
$L^p$-spaces,} J. Geom. Anal. 6 (1996), no. 3, 365--383 (1997).

\bibitem{AQ} M. L. Agranovsky and E. T. Quinto, {\em Injectivity sets for the Radon transform over circles
and complete systems of radial functions,} J. Funct. Anal.,  139  (1996),  no. 2, 383--414.

\bibitem{AR} M. L. Agranovsky and R. Rawat, {\em Injectivity sets for spherical means on the Heisenberg group,}
J. Fourier Anal. Appl., 5 (1999), no. 4, 363--372.

\bibitem{AK}  G. Ambartsoumian and P. Kuchment, {\em On the injectivity of the circular Radon transform,}
 Inverse Problems 21 (2005), no. 2, 473--485.

\bibitem{AVZ} M. L. Agranovsky, V. V. Volchkov and L. A. Zalcman, {\em Conical uniqueness sets for the
spherical Radon transform,} Bull. London Math. Soc. 31 (1999), no. 2, 231-–236.

\bibitem{BP} E. Binz and S. Pods, {\em The geometry of Heisenberg groups. With applications in signal theory, optics,
quantization, and field quantization. With an appendix by Serge Preston,}
 Mathematical Surveys and Monographs, 151. AMS, Providence, RI, 2008.

\bibitem{CCG} O. Calin, D. Chang and P. Greiner, {\em Geometric analysis on the Heisenberg group and its generalizations,}
AMS/IP Studies in Advanced Mathematics, 40. American Mathematical Society, Providence, RI; International Press,
Somerville, MA, 2007.

\bibitem{CDPT} L. Capogna, D. Danielli, S. D. Pauls and J. T. Tyson, {\em An introduction to the Heisenberg
group and the sub-Riemannian isoperimetric problem,} Progress in Mathematics, 259. Birkhauser Verlag, Basel, 2007.

\bibitem{CH} R. Courant and D. Hilbert, {\em Methods of Mathematical Physics,}  Vol. 2.

\bibitem{H} R. Howe, {\em Roger On the role of the Heisenberg group in harmonic analysis,}
 Bull. Amer. Math. Soc. (N.S.) 3 (1980), no. 2, 821–843.

\bibitem{N} R. Narasimhan, {\em Analysis on real and complex manifolds,} North-Holland Publishing Co.,
Amsterdam, 1985.

\bibitem{NRR} E. K. Narayanan, R. Rawat and S. K.  Ray, {\em Approximation by $K$-finite functions in
$L\sp p$ spaces,} Israel J. Math.  161  (2007), 187-207.

\bibitem{NT1} E. K. Narayanan and S. Thangavelu, {\em Injectivity sets for spherical means on the
Heisenberg group,} J. Math. Anal. Appl. 263 (2001), no. 2, 565-579.

\bibitem{PS} V. Pati and A. Sitaram, {\em Some questions on integral geometry on Riemannian manifolds,}
 Ergodic theory and harmonic analysis (Mumbai, 1999). Sankhya Ser. A 62 (2000), no. 3, 419-424.

\bibitem{RS} R. Rawat and A. Sitaram, {\em Injectivity sets for spherical means on $\mathbb R^n$
and on symmetric spaces,} J. Fourier Anal. Appl. 6 (2000), no. 3, 343-348.

\bibitem{Ru} W. Rudin, {\em Function theory in the unit ball of ~$\mathbb C^n$,} Springer-Verlag,
 New York-Berlin, 1980.

\bibitem{ST} G. Sajith and S. Thangavelu, {\em On the injectivity of twisted spherical means on ~$\mathbb C^n$,}
 Israel J. Math. 122 (2001), 79-92.

 \bibitem{Sri} R. K. Srivastava, {\em Sets of injectivity for weighted twisted spherical means and support theorems,}
J. Fourier Anal. Appl.,  18 (2012), no. 3, 592-608.

\bibitem{Sri2} R. K. Srivastava, {\em Coxeter system of lines are sets of injectivity for the twisted
spherical means on $\mathbb C$,} (communicated).
DOI:~{\color{blue}\href{http://arxiv.org/find/all/1/au:+srivastava_r_k/0/1/0/all/0/1}{arXiv:1103.4571v1}}

\bibitem{T} S. Thangavelu, {\em An introduction to the uncertainty principle},
Prog. Math. 217, Birkhauser, Boston (2004).

\end{thebibliography}
\end{document}